\documentclass[11pt]{article}
\input{Macros.sty}

\begin{document}

\begin{center}
\textbf{\LARGE B\'ezout's theorem for abelian varieties}
\bigskip

\textit{by}
\bigskip

{\Large Olivier Debarre\footnote{Supported by the European Research Council under the European Union's Horizon 2020 research and innovation programme (ERC-2020-SyG-854361-HyperK)} and Ben Moonen}
\end{center}
\vspace{6mm}

{\small 

\noindent
\begin{quoting}
\textbf{Abstract.} Let $X$, $Y$ be irreducible closed subvarieties  of an absolutely simple abelian variety of dimension~$g$  over a field. If  $\dim(X) + \dim(Y) \leq g$, we prove that the addition morphism $X\times Y \to X+Y$ is semismall. As a consequence, we deduce that if $\dim(X) + \dim(Y) \geq g$, the subvarieties $X$ and $ Y$ must meet (B\'ezout's theorem). If we drop the assumption that the abelian variety is  absolutely simple, we prove that B\'ezout's theorem still holds if $X$ satisfies a nondegeneracy condition. These results were previously known only in characteristic zero.

Our proof of the semismallness statement is based on the theory of perverse sheaves: using results of Kr\"amer and Weissauer, we prove  that for perverse sheaves $K$ supported on~$X$, and~$L$ supported on~$Y$, the convolution product $K * L$ is again perverse.
\medskip

\end{quoting}

} 
\vspace{2mm}

\frenchspacing

\section{Introduction}

Any two closed subvarieties of a projective space~$\bbP^g_k$ over a field~$k$, whose dimensions add up to at least~$g$, meet. The classical proof of this fundamental fact ultimately rests on Krull's Hauptidealsatz; it is thus independent of the characteristic of $k$. Consider now the analogous question for closed subvarieties of an abelian variety $A$ of dimension~$g$ over a field~$k$: do subvarieties of $A$ whose dimensions add up to at least~$g$ always meet? For this question to have a positive answer, it is clear that one needs to assume that $A$ is simple. As we shall discuss in Remark~\ref{rem:simpleAV}, this assumption is not strong enough, in general, and in many of our results we shall assume that $A$ is absolutely simple (that is, the extension of~$A$ to an algebraic closure of~$k$ has no nonzero proper abelian subvarieties). 

Assuming that $A$ is absolutely simple, the above question does have a positive answer when $\charact(k) = 0$. For $k=\bbC$, this was first proved by Barth~\cite{Barth}. A simple proof for arbitrary fields~$k$ with $\charact(k) = 0$ was given by D.~Prasad~\cite{Prasad}, and independently also by one of us~\cite{Debarre-Conn} (the arguments are similar and we review them in Section~\ref{sec:background}). 

When $\charact(k) > 0$, the arguments used in~\cite{Prasad} and  \cite{Debarre-Conn} break down and Prasad asked whether the question still has an affirmative answer.  We prove that it does: this is a consequence (see Corollary~\ref{cor:codim}) of our first main result, which can be stated as follows (see Corollary~\ref{cor:dimZ}).

\begin{theorem}\label{thmA}
Let $A$ be an absolutely simple abelian variety of dimension~$g$ over a field. Let $X_1,\ldots,X_r$ be closed subvarieties of~$A$ and define $Z \coloneqq X_1 + \cdots + X_r$. We have
\[
\dim(Z) = \min\biggl\{g,\sum_{i=1}^r\; \dim(X_i)\biggr\}\, .
\]  
\end{theorem}

We deduce Theorem~\ref{thmA} from the following result (see Theorem~\ref{thm:semismall}).

\begin{theorem}\label{thmB}
Let $A$ be an absolutely simple abelian variety of dimension~$g$ over a field. Let $X$ and~$Y$ be irreducible closed subvarieties of~$A$ and define $Z \coloneqq X+Y \subset A$. If $\dim(X) + \dim(Y) \leq g$, the addition morphism $X\times Y \to Z$ is semismall.
\end{theorem}

Recall that if $S$ and~$T$ are varieties over a field, a proper morphism $f \colon S \to T$ is said to be semismall if for every integer~$n\geq 0$, the dimension of the closed subset $T_n \coloneqq \bigl\{t \in T \bigm| \dim(f^{-1}(t)) \geq n\bigr\} $ of $ T$ is at most $\dim(S) -2n$. If $S$ is irreducible and $f$ is semismall and surjective, then $f$ is generically finite.

\subsection{}
Since there is a strong connection between semismallness of a morphism and the theory of perverse sheaves (see for instance \cite[Section~4.2]{dCM}), it may not come as a surprise that our proof of Theorem~\ref{thmB} relies on a result about perverse sheaves. More precisely, suppose we have an absolutely simple abelian variety~$A$ of dimension~$g$ and irreducible closed subvarieties $X$, $Y \subset A$ with $\dim(X) + \dim(Y) \leq g$. We prove, under mild technical assumptions, that for perverse sheaves~$K$ supported on~$X$, and~$L$ supported on~$Y$, the convolution product $K * L$ is again perverse; see Proposition~\ref{prop:K*Lperv}. It then still requires some work to deduce Theorem~\ref{thmB}, but the main idea for this is quite simple; we explain it in~\ref{subsec:ideaproof}.

The proof of Proposition~\ref{prop:K*Lperv} is based on the work of Kr\"amer and Weissauer in \cite{KW} and~\cite{WeissFF}. For our applications, we need to refine some of their results, especially when we work in characteristic~$p$. One aspect of this is that our result is valid without any assumptions on the semisimplicity of the perverse sheaves involved. Also, in characteristic~$p$, our result is valid whenever the perverse sheaves are defined over some ring of finite type over~$\bbF_p$. In this setting, some of the technical properties that we need are deduced from the analogous results over finite fields using `spreading out' techniques, as in the work of Drinfeld~\cite{Drinfeld} and Weissauer~\cite{WeissFF}. We give an exposition of this using the theory of relative perverse sheaves that was recently developed by Hansen and Scholze~\cite{HansenScholze}.

\subsection{}
The results about perverse sheaves that we need are discussed in Section~\ref{sec:ConvPerv}, and the proofs of Theorems~\ref{thmA} and~\ref{thmB} are given in Section~\ref{sec:AddSimpleAV}. Finally, in Section~\ref{sec:AddArbitraryAV}, we give a partial extension of our main results to arbitrary abelian varieties. In particular, we show that if we drop the assumption that $A$ is absolutely simple, Theorem~\ref{thmA} is still valid if at least $r-1$ of the subvarieties~$X_i$ are geometrically nondegenerate in the sense of Ran's paper~\cite{Ran}; see Corollary~\ref{cor:X1Xr}.

\subsection{Acknowledgment.}
We thank the referees for their comments and suggestions for improvements of the exposition. We thank Beat Zurbuchen for pointing out a mistake in an earlier version of the paper.

\subsection{Notation and conventions}\label{subsec:NotConv}
\begin{enumerate}
\item By a variety over a field~$k$, we mean a reduced and separated $k$-scheme of finite type.

\item\label{conv:subvar} Throughout, $A$ denotes an abelian variety over a field~$k$, with addition map $\sigma\colon A \times A \to A$ and origin $0 \in A(k)$. 

\item If $X$, $Y \subset A$ are closed subvarieties, we write $X+Y$ for the image of $X \times Y$ under~$\sigma$. Similarly for the sum of more than two subvarieties.

\item\label{conv:tgtbunA} The tangent bundle of~$A$ is identified with $A \times T_{A,0}$.

\item If $a\in A(k)$, we let $t_a \colon A\to A$ be the morphism given by translation by~$a$. If $X \subset A$ is a closed subvariety, we denote by $\Stab(X) \subset A$ the stabiliser of~$X$, that is, the subgroup scheme of~$A$ consisting of the points $a \in A$ such that $t_a(X) = X$. If $k$ is perfect, the reduced identity component of~$\Stab(X)$ is an abelian subvariety of~$A$. 

\item If $X$ is a noetherian topological space, $\dim(X)$ denotes its Krull dimension, which equals the maximum of the (Krull) dimensions of the irreducible components of~$X$.
\end{enumerate}

\section{Background and some elementary cases}\label{sec:background}

\subsection{}
We first give a brief sketch, following~\cite{Prasad}, of the proof of Theorem~\ref{thmA} for a simple complex abelian variety~$A$ of dimension $g$. The essential case to consider is when we have irreducible closed subvarieties $X$, $Y \subset A$ with $\dim(X) + \dim(Y) \leq g $; in that case, we want to show that the addition map $X\times Y \to A$ is generically finite onto its image $Z \coloneqq X+Y$. If this is not true, all fibres of the map $X \times Y \to Z$ have positive dimension and $Z \subsetneq A$. Choose a curve $F \subset X\times Y$ such that $\sigma(F) = \{z\}$ for some regular point~$z$ of~$Z$. The projection map $\pr_X \colon X\times Y \to X$ induces an isomorphism $F \isomarrow C \coloneqq \pr_X(F)$, with inverse given by $x \mapsto (x,z-x)$. By construction, $T_{X,x} + T_{Y,z-x} \subset T_{Z,z} \subsetneq T_{A,0}$ for all $(x,z-x) \in F$ (note convention~\ref{subsec:NotConv}\ref{conv:tgtbunA}). Hence, we have inclusions
\[
T_{C,x} \subset T_{X,x} \subset T_{Z,z} \subsetneq T_{A,0}
\] 
for all $x \in C$. Finally, one shows (see \cite[Lemma~1]{Prasad}) that in a simple complex abelian variety~$A$, one cannot have a regular (locally closed) curve $C \subset A$ and a subspace $V \subsetneq T_{A,0}$ such that $T_{C,x} \subset V$ for all $x \in C$.

\subsection{}
It was pointed out by Abramovich that the last step in the proof (that is, \cite[Lemma~1]{Prasad}) breaks down over fields of positive characteristic. (In~\cite{Debarre-Conn}, the property described in~(1.9) of that paper is no longer true in positive characteristic.) More specifically: let $k$  be an algebraically closed field of  characteristic $p>0$ and let~$A$ be a simple abelian variety over~$k$. Choose a nontrivial infinitesimal subgroup scheme $N \subset A$ (for instance, any nontrivial local subgroup scheme of the $p$-kernel group scheme~$A[p]$). The quotient morphism $q\colon A \to B = A/N$ is then a purely inseparable isogeny  and $T_{N,0} = \Ker(T_{q,0}\colon T_{A,0} \to T_{B,0})$ is nonzero. If $C \subset A$ is a regular curve passing through the origin such that $T_{C,0} \not\subset T_{N,0}$, its image $q(C) \subset B$ is a generically regular curve  and its regular locus $C^\prime\subset B $ is a curve that has the property that $T_{C^\prime,x} \subset \Image(T_{q,0}) \subsetneq T_{B,0}$ for all $x \in C^\prime$.

\subsection{}
Some low dimensional cases of Theorem~\ref{thmA} can be handled by elementary arguments (in arbitrary characteristic). These arguments will not be used in the rest of the paper and are only included to illustrate how one may try to argue. We consider the situation where $k$ is algebraically closed, $A$ is a simple  abelian variety of dimension $g$ over~$k$, and $X$, $Y \subset A$ are irreducible closed subvarieties, of respective dimensions $d$ and~$e$. We assume $d+e\leq g$ and we set $Z \coloneqq X+Y$. The goal is to show that $\dim(Z) = d+e$. We will do that when $d=1$, or when $d=2$ and either $e\in\{2,3\}$ or~$g>2e$. Note that the stabiliser of any closed subvariety $W \subsetneq A$ is finite, because $A$ is simple.
\medskip

\noindent
(1) If $\dim(Z) = e$, then $X + Y = a+Y$ for every $a\in X(k)$, hence $X-a \subset \Stab(Y)$. This is possible only if $X$ is a point. In particular, if $\dim(X) = 1$ then $\dim(Z) = e+1$.
\medskip

\noindent
(2) Assume $d = 2$ and $\dim(Z) = e+1$. For $z \in Z(k)$, define $F_z \coloneqq \sigma^{-1}(z)$ and let $C_z \coloneqq\pr_X(F_z)$, which is a closed subscheme of~$X$. The projection map induces an isomorphism $F_z \isomarrow C_z$ whose inverse is given on points by $w \mapsto (w,z-w)$.

Consider the morphism $\psi \colon X\times Y \times Y \to Z \times Z$ given by $\psi(x,y_1,y_2) = (x+y_1,x+y_2)$, and let $W \subset Z\times Z$ be its image. For $z \in Z(k)$, we have  $\dim(C_z) \geq 1$, hence $C_z+Y=Z$  by case~(1), and $\{z\}\times Z \subset W$. This implies that $\psi$ is surjective; therefore, for every $z_1, z_2 \in Z(k)$, there exist points $x\in X(k)$ and $y_1$, $y_2 \in Y(k)$ with $z_i = x+y_i$. It follows that $Z-Z \subset Y-Y$; but $Z-Z = X-X + Y-Y$, so we get   $X-X \subset \Stab(Y-Y)$, which is possible only if $Y-Y = A$. In particular, if $d = 2$ and $e < \tfrac{g}{2}$, we have $\dim(Z) = e+2$.
\medskip

\noindent
(3) Assume $d = 2$ and $e \in\{2,3\}$. Suppose $\dim(Z) = e+1$. The argument in case~(2) shows that $Y-Y = A$, and either $(e,g) = (2,4)$ or $e=3$ and $g\in\{5,6\}$. The fibre of the difference map $\delta \colon Y \times Y \to A$ above a point $a\in A(k)$ is isomorphic to $Y \cap (a+Y)$. Let  $L \subset A$ be the closed subset of~$A$ where this fibre has dimension at least~$e-1$. We have
\[
2e = \dim(Y\times Y)\geq \dim\bigl(\delta^{-1}(L)\bigr)\geq \dim (L)+e-1\, ,
\]
so that $\dim(L)\leq 1+e< d+e \leq g$ and $L\subsetneq A$. Therefore, $\delta^{-1}(L) \subsetneq Y\times Y$ and the inequality we have found can be improved to $\dim(L)\le e$.

Let $C$, $C^\prime \subset X$ be curves such that $C^\prime$ is not a translate of~$C$. By case~(1), we have $C+Y =  C^\prime + Y=Z$ and the surjective addition maps $C\times Y \to Z$ and $C^\prime \times Y \to Z$ are generically finite. Set $V \coloneqq(C\times Y) \times_Z (C^\prime \times Y)$; then $\dim(V) \geq e+1$. Let $W \subset C\times Y \times C^\prime$ be the image of~$V$ under projection onto the first three factors. Then $V \isomarrow W$, hence $\dim(W) \geq e+1$. Consider the projection map $\pi \colon W \to C\times C^\prime$. The fibre over a point $(c,c^\prime)$ is $e$-dimensional if and only if $c-c^\prime \in \Stab(Y)$. By our assumption on $C$ and~$C^\prime$, this happens at only finitely many points $(c,c^\prime)$. Hence $\pi$ is surjective and all its fibres have dimension at least~$e-1$. This implies that $\bigl((c-c^\prime) + Y\bigr) \cap Y$ has dimension at least~$e-1$ for all $(c,c^\prime) \in C\times C^\prime$, that is, $C-C^\prime \subset L$. Since this holds for arbitrary curves $C$, $C^\prime \subset X$ with $C \not\cong C^\prime$, we get $X-X \subset L$, hence $\dim(X-X)\leq \dim(L)\leq e$.

Applying~(2) to the subvarieties $X$ and~$-X$, we find that $\dim(X-X) \leq 3$ is possible only if $X-X = A$; but then, $\dim(X-X)\leq e$ implies $g\leq e$, which in the cases that we are considering does not hold. This settles the case where $d = 2$ and $e \in\{2,3\}$.
\medskip

\noindent
It seems  difficult to obtain a proof of Theorem~\ref{thmA} in general using such elementary arguments. Already the cases where $(d,e) = (2,4)$ and $g\in \{6,7,8\}$ pose  a challenge.

\section{Convolution product and perverse sheaves}\label{sec:ConvPerv}

The proofs of the results that will be given in the next two sections are based on Proposition~\ref{prop:K*Lperv} below, which concerns the convolution product of perverse sheaves on abelian varieties. The key ideas for this proposition come from the work of Kr\"amer and Weissauer in \cite{KW} and~\cite{WeissFF}. In order to obtain the result over finitely generated fields of characteristic~$p$, one needs some `spreading out' arguments as in the work of Drinfeld~\cite{Drinfeld}, which were also used in~\cite{WeissFF}. We give a streamlined exposition based on the work of Hansen and Scholze~\cite{HansenScholze}, which we summarize in~\ref{ssec:spreadingout}--\ref{ssec:ULAprops}. 

One of the main tools in this section is that a perverse sheaf on an abelian variety cannot have negative Euler characteristic. In Proposition~\ref{prop:chiPnonneg}, we prove this over an arbitrary base field using an argument similar to the one that, working over~$\bbC$, is given in~\cite{FranKapr}. The second important fact is that, under suitable technical assumptions, a perverse sheaf~$P$ with $\chi(P) = 0$ has a nontrivial stabilizer; see Proposition~\ref{prop:dimStab>0}. Together with Proposition~\ref{prop:chiP=0}, this gives us the ingredients needed in the proof of Proposition~\ref{prop:K*Lperv}.

\subsection{}
Throughout the discussion, $\ell$ denotes a prime number which is assumed to be invertible on all schemes that we consider. We shall almost exclusively work with $\Qlbar$ as a coefficient ring but, in Lemma~\ref{lem:chi=chi} and the proof of Proposition~\ref{prop:chiPnonneg}, we also need to consider other choices; so for now, let $\Lambda$ be any commutative Hausdorff topological ring (or any condensed commutative ring). 

If $X$ is any quasi-compact quasi-separated scheme on which $\ell$ is invertible, we write $\uD_{\cstr}(X,\Lambda)$ for the triangulated category of constructible complexes of $\Lambda$-modules on the pro-\'etale site~$X_\proet$ of~$X$, as in~\cite{HemoEtAl}. For $E$ an algebraic field extension of~$\bbQ_\ell$ and $\Lambda \in \{E,\cO_E\}$, this is the same as the category defined in \cite[Definition~6.8.8]{BS}. As discussed in~\cite{BS}, the formalism we use is compatible with the $\ell$-adic formalism that is used in some older references. Every object $K$ of $\uD_{\cstr}(X,\Qlbar)$ is bounded and the cohomology sheaves $\cH^i(K)$ are constructible. By the support of~$K$, we mean the union of the supports of the~$\cH^i(K)$.

Let $k$ be a field with algebraic closure~$\kbar$ and let $X$ be a quasi-projective scheme over~$k$. For $K \in \uD_{\cstr}(X,\Lambda)$, all cohomology groups $H^i(X_{\kbar},K)$, as well as the cohomology groups with compact support $H^i_{\comp}(X_{\kbar},K)$, are $\Lambda$-modules of finite type, which are nonzero only for finitely many integers~$i$. If $\Lambda$ is a field, define
\[
\chi(K) \coloneqq \sum_{i\in \bbZ}\; (-1)^i \dim_\Lambda\bigl(H^i(X_{\kbar},K)\bigr)\, ,\qquad
\chi_{\comp}(K) \coloneqq \sum_{i\in \bbZ}\; (-1)^i \dim_\Lambda\bigl(H^i_{\comp}(X_{\kbar},K)\bigr)\, .
\]
If $K \to L \to M \to$ is a distinguished triangle, $\chi(K) + \chi(M) = \chi(L)$ and $\chi_{\comp}(K) + \chi_{\comp}(M) = \chi_{\comp}(L)$.

\subsection{}
Let $A$ be an abelian variety of dimension~$g$ over a field~$k$, with addition map $\sigma \colon A \times A \to A$. For $K$, $L \in \uD_{\cstr}(A,\Qlbar)$, we write $K * L \coloneqq R\sigma_*(K\boxtimes L)$ for their convolution product. (The external tensor product is defined by $K \boxtimes L \coloneqq \pr_1^*(K) \otimes^{\mathrm{L}} \pr_2^*(L)$, where $\pr_1$, $\pr_2 \colon A\times A \to A$ are the projection maps.) The convolution product makes~$\uD_{\cstr}(A,\Qlbar)$ into a $\Qlbar$-linear rigid symmetric monoidal category. (For rigidity, see~\cite{WeissRigid}.) The identity object is the skyscraper sheaf~$\IC_{\{0\}}$ at the origin (which is the intersection complex of $\{0\} \subset A$, as defined in \ref{subsec:PervNotat}; whence the notation). The dual of an object~$K$ is $(-\id_X)^* \bbD(K)$, where~$\bbD$ is the Verdier duality functor. We have functorial isomorphisms
\begin{equation}\label{eq:D(K*L)}
\bbD(K*L) \cong \bbD(K) * \bbD(L)\, .
\end{equation}

\subsection{}\label{subsec:PervNotat}
Let $\Perv(A) \subset \uD_{\cstr}(A,\Qlbar)$ be the subcategory of perverse sheaves. If $X$ is a closed subvariety of~$A$, we identify $\Perv(X)$ with the category of perverse sheaves on~$A$ that have support contained in~$X$. If $X$ is irreducible, the intersection complex of~$X$, which is a selfdual simple perverse sheaf, is defined as $\IC_X = j_{!*}\Qlbar[d]$, where $d = \dim(X)$ and where $j\colon U \hookrightarrow X$ is the inclusion of a dense open subset of~$X$ such that $(U_{\kbar})_{\red}$ is smooth over~$\kbar$. (By \cite[Corollaire~17.15.13]{EGAIV}, there exists such an~$U$, and it is a basic fact that $\IC_X$ is independent of which~$U$ with this property we choose.) For an arbitrary closed subvariety~$X$, we define~$\IC_X$ to be the direct sum of the intersection complexes of its irreducible components.

If $0 \to P^\prime \to P \to P^\pprime \to 0$ is a short exact sequence in $\Perv(A)$, there is a distinguished triangle $P^\prime \to P \to P^\pprime \to$ in $\uD_{\cstr}(A,\Qlbar)$. Hence the Euler characteristic~$\chi(P)$ is additive in short exact sequences. For a perverse sheaf~$P$ which is supported on a subset of dimension~$d$, we have $H^i(A_{\kbar},P) = 0$ for all $i>d$. 

The category $\Perv(A)$ is artinian and noetherian, so a perverse sheaf~$P$ on~$A$ has a socle filtration
\[
\soc_\bullet(P):\qquad 0 = \soc_0(P) \subset \soc_1(P) \subset \soc_2(P) \subset \cdots \subset \soc_t(P) = P
\]
(for some $t\geq 0$), which is defined inductively by the rule that $\soc_{i+1}(P)/\soc_i(P) \subset P/\soc_i(P)$ is the maximal semisimple subobject of~$P/\soc_i(P)$. We refer to the number~$t$ as the socle length of~$P$.

\subsection{}\label{subsec:LpsiDef}
If $\psi\colon \pi_1(A,0) \to \ol\bbQ^\times_\ell$ is a continuous character, we denote by~$\bbL_\psi$ the corresponding smooth $\Qlbar$-sheaf of rank~$1$ on~$A$. We have an isomorphism
\begin{equation}\label{eq:DLpsi}
\bbD(\bbL_\psi) \cong \bbL_{\psi^{-1}}[2g]\, .
\end{equation}
With the notation $K_\psi = K \otimes \bbL_\psi$, we have
\begin{equation}\label{eq:Kpsi*Lpsi}
K_\psi * L_\psi \cong (K*L)_\psi
\end{equation}
for all $K$ and~$L$ in $\uD_{\cstr}(A,\Qlbar)$.

\subsection{}
One of the key points in the theory is that, for every perverse sheaf~$P$ on an abelian variety~$A$ over a field~$k$, we have $\chi(P) \geq 0$. For $k = \bbC$, this was proven in \cite[Corollary~1.4]{FranKapr}, and the result for arbitrary base fields of characteristic~$0$ can be reduced to that case. Over fields of positive characteristic, the result was proven in~\cite{WeissFF} under the assumption that $P$ can be defined over a finitely generated field. Using T.~Saito's results on characteristic cycles of perverse sheaves,  essentially the same argument as in~\cite{FranKapr} gives the result in full generality; see Proposition~\ref{prop:chiPnonneg} below. Saito's results in~\cite{Saito} are stated for perverse sheaves with coefficients in a finite ring. We can apply these results because of the following fact, here stated only in the setting in which we need it.

\begin{lemma}\label{lem:chi=chi}
Let $X$ be a quasi-projective variety over an algebraically closed field~$k$. Let $E$ be a finite field extension of~$\bbQ_\ell$ with ring of integers~$\cO_E$ and residue field~$\kappa_E$. For $K \in \uD_{\cstr}(X,\cO_E)$, let $K[\frac{1}{\ell}]$ denote its image in~$\uD_{\cstr}(X,\Qlbar)$  and let $\bar{K}$ denote the image of~$K$ in~$\uD_{\cstr}(X,\kappa_E)$. Then $\chi_{\comp}\bigl(K[\frac{1}{\ell}]\bigr) = \chi_{\comp}(\bar{K})$.
\end{lemma}

\begin{proof}
Write $S \coloneqq \Spec(k)$. For $\Lambda \in \{\Qlbar,\cO_E,\kappa_E\}$, the category $\uD_{\cstr}(S,\Lambda)$ is equivalent to the category $\Perf_\Lambda$ of perfect complexes of $\Lambda$-modules; see \cite[Lemma~4.1]{HemoEtAl}. In particular, in each of the three cases, the Grothendieck group of $\uD_{\cstr}(S,\Lambda)$ is isomorphic to~$\bbZ$. We have a commutative diagram
\[
\begin{tikzcd}
\uD_{\cstr}(X,\Qlbar) \ar[d,"R\Gamma_{\comp}"] & \uD_{\cstr}(X,\cO_E) \ar[l] \ar[r] \ar[d,"R\Gamma_{\comp}"] & \uD_{\cstr}(X,\kappa_E)\ar[d,"R\Gamma_{\comp}"]\\
\uD_{\cstr}(S,\Qlbar) & \uD_{\cstr}(S,\cO_E) \ar[l] \ar[r] & \uD_{\cstr}(S,\kappa_E)
\end{tikzcd}
\]
and the assertion follows by considering the induced maps on Grothendieck groups. 
\end{proof}

\begin{proposition}\label{prop:chiPnonneg}
Let $A$ be an abelian variety over a field~$k$. Then $\chi(P) \geq 0$ for every perverse sheaf~$P$ on~$A$.
\end{proposition}

\begin{proof}
We may assume that $k$ is algebraically closed. For every perverse sheaf~$P$ with $\Qlbar$-coefficients, there is a finite field extension $\bbQ_\ell \subset E$ such that $P$ is of the form $Q[\frac{1}{\ell}]$ (notation as in Lemma~\ref{lem:chi=chi}) for a perverse sheaf~$Q$ in $\uD_{\cstr}(A,\cO_E)$, and such that the image $\bar{Q}$ of~$Q$ in $\uD_{\cstr}(A,\kappa_E)$ is again perverse (this is a special case of \cite[Proposition~6.11]{HansenScholze}). By Lemma~\ref{lem:chi=chi}, it suffices to show that $\chi(\bar{Q}) \geq 0$. 

Let $T^\vee \!A$ be the cotangent bundle of~$A$. Let $\CC(\bar{Q})$ be the characteristic cycle of~$\bar{Q}$, as defined by Saito in~\cite[Section~5.3]{Saito}; it is an integral linear combination $\CC(\bar{Q}) = \sum_j\, m_j  [C_j]$ of irreducible closed conical subsets $C_j \subset T^\vee \!A$ of dimension~$g$. By~\cite[Proposition~5.14]{Saito}, $\CC(\bar{Q})$ is effective, that is, $m_j \geq 0$ for all~$j$. Furthermore, by~\cite[Theorem~7.13]{Saito}, the Euler characteristic~$\chi(\bar{Q})$ equals the intersection number of~$\CC(\bar{Q})$ with the zero section of~$T^\vee\! A$. Because $T^\vee \!A = A \times T^\vee_{A,0}$, we have a projection map $\pr\colon T^\vee \!A \to T^\vee_{A,0}$ (which is a proper morphism) and the zero section of~$T^\vee A$ is the inverse image of $0 \in T^\vee_{A,0}$. If we denote by~$d_j$ the degree of the morphism $\pr|_{C_j} \colon C_j \to T^\vee_{A,0}$ (the Gauss map of~$C_j$), the projection formula gives $\chi(\bar{Q}) = \sum_j\; m_j  d_j$. Because all $m_j$ and~$d_j$ are nonnegative, this gives the assertion. 
\end{proof}

\begin{remark}\label{rem:Notation0}
We shall frequently work in a setting where we have an abelian variety~$A_0$ over a base field~$k_0$ which is contained in an algebraically closed field~$k$. In such a situation, the notational convention is that a subscript~$0$ is used for objects defined over~$k_0$, and for the corresponding objects over~$k$ the subscript~$0$ is omitted. If $k_0 \subset k_1 \subset k$ is an intermediate field, a subscript~$1$ is used for objects over~$k_1$.
\end{remark}

\begin{lemma}\label{lem:SSPervFq}
Let $A_0$ be an abelian variety over a field~$k_0$. Set $A = A_0 \otimes_{k_0} k$, where $k$ is an algebraic closure of~$k_0$. Let $P_0$ be a perverse sheaf on~$A_0$ and let $P$ be its pullback to~$A$.
\begin{enumerate}
\item\label{P0ssPss} If $P_0$ is semisimple, so is~$P$.

\item\label{Isotyp/k1} Suppose $P_0$ is semisimple  and let 
\begin{equation}\label{eq:isotdec}
P = \bigoplus_\alpha P^{(\alpha)}
\end{equation}
be the isotypical decomposition of~$P$. Then this decomposition is defined over a finite extension~$k_1$ of~$k_0$ inside~$k$; by this, we mean that if~$P_1$ is the pullback of~$P_0$ to $A_1 = A_0 \otimes_{k_0} k_1$, there is a decomposition $P_1 = \bigoplus_\alpha P_1^{(\alpha)}$ that after pullback to~$A$ gives back~\eqref{eq:isotdec}.
\end{enumerate}
\end{lemma}

\begin{proof}
It suffices to prove this under the assumption that $P_0$ is simple. In this case, there is an irreducible subvariety $j \colon V_0 \hookrightarrow A_0$ such that $V_{\red}$ is smooth over~$k$, and an irreducible lisse $\ell$-adic sheaf~$\cE_0$ on~$V_0$ such that $P_0 \cong j_{!*}\cE_0\bigl[\dim(V_0)\bigr]$. There is a finite Galois extension $k_0 \subset k_1$ inside~$k$ such that all irreducible components of~$V_1$ are geometrically irreducible. Let $W^{(1)},\ldots,W^{(n)}$ be the irreducible components of~$V_1$, which are pairwise disjoint because $V_{\red}$ is smooth over~$k$, and let $\cE^{(\nu)}$ be the pullback of~$\cE_0$ to~$W^{(\nu)}$, which is a semisimple lisse $\ell$-adic sheaf. Replace~$V_0$ by any of the~$W^{(\nu)}$ and replace~$\cE_0$ by any of the simple factors of~$\cE^{(\nu)}$; then we are in the same situation as before but with the additional property that~$V_0$ is geometrically irreducible. Choose a base point $v \in V(k)$  and let $E$ be the fibre of~$\cE_0$ at~$v$. Then $\cE_0$ corresponds to an irreducible representation $\rho \colon \pi_1(V_0,v) \to \GL(E)$. Assertion~\ref{P0ssPss} follows from the fact that the restriction~$\rho^{\geo}$ of~$\rho$ to the geometric fundamental group~$\pi_1(V,v)$ is semisimple, because $\pi_1(V,v)$ is a closed normal subgroup of~$\pi_1(V_0,v)$. Assertion~\ref{Isotyp/k1} follows from the fact that there is a finite Galois extension $k_0 \subset k_1$ inside~$k$ such that $\pi_1(V_1,v) \subset \pi_1(V_0,v)$ preserves the isotypical decomposition of~$\rho^{\geo}$ (cf.\ the proof of~\cite[Proposition~5.3.9]{BBD}).
\end{proof}

\subsection{}\label{ssec:spreadingout}
Let $A_0$ be an abelian variety over a field~$k_0$ which is a finitely generated extension of~$\bbF_p$. We discuss some techniques that allow us to reduce statements about perverse sheaves on~$A_0$ to the case where $k_0$ is a finite field. Our arguments are inspired by the work of Drinfeld~\cite{Drinfeld} and Weissauer~\cite{WeissFF}  but we shall rephrase these arguments using the theory of relative perverse sheaves developed by Hansen and Scholze in~\cite{HansenScholze}.

Let $\kappa$ be the algebraic closure of~$\bbF_p$ inside~$k_0$, which is a finite field. There exists a geometrically integral variety~$S$ over~$\kappa$ whose function field is~$k_0$ and such that $A_0$ extends to an abelian scheme $\pi \colon \cA \to S$. After shrinking~$S$, we may assume~$S$ is smooth over~$\kappa$;  from now on, we assume this. If~$\bar{s}$ is a geometric point of~$S$, let $\cA_{\bar{s}}$ be the fibre of~$\pi$ over~$\bar{s}$. Furthermore, let $\eta$ be the generic point of~$S$, so that $\cA_\eta = A_0$, and let $j_\eta \colon \cA_\eta \hookrightarrow \cA$ be the inclusion.

For the definition of when an object $\cK \in \uD_{\cstr}(\cA,\Qlbar)$ is called universally locally acyclic (abbreviated to ULA) relative to~$S$, we refer to \cite[Definition~3.2]{HansenScholze}. In \cite[Theorem~4.4]{HansenScholze}, it is shown that this agrees with the more traditional definition. We write $\uD^{\ULA}(\cA/S,\Qlbar)$ for the full subcategory of $\uD_{\cstr}(\cA,\Qlbar)$ consisting of the objects that are ULA relative to~$S$. On this category, Hansen and Scholze define a relative perverse $t$-structure, whose heart $\Perv^{\ULA}(\cA/S)$ is artinian and noetherian (see \cite[Theorem~6.8]{HansenScholze}). The relative perverse sheaves are the objects~$\cP$ of $\uD^{\ULA}(\cA/S)$ with the property that $\cP|_{\cA_{\bar{s}}}$ is perverse in the usual sense, for all geometric points~$\bar{s}$ of~$S$.

\subsection{}\label{ssec:ULAprops}
We list some properties that we need. Notation and assumptions are as above.

\subsubsection{}\label{sssec:jeta*}
By \cite[Theorem~6.8]{HansenScholze}, the functor $j_\eta^* \colon \Perv^{\ULA}(\cA/S) \to \Perv(\cA_\eta)$ is exact and fully faithful, and its essential image is stable under subquotients. Furthermore, if $i_s \colon \cA_s \hookrightarrow \cA$ is the inclusion of a fibre, the functor  $i_s^* \colon \uD^{\ULA}(\cA/S,\Qlbar) \to \uD_{\cstr}(\cA_s,\Qlbar)$ is $t$-exact (this follows from  \cite[Theorem~6.1(ii)]{HansenScholze}).

\subsubsection{}\label{sssec:genericULA}
If $\cK$ is an object of $\uD_{\cstr}(\cA,\Qlbar)$, there is a dense open subset $U \subset S$ such that the restriction of~$\cK$ to $\cA_U = \pi^{-1}(U)$ is an object of $\uD^{\ULA}(\cA_U/U)$. This is a result of Deligne, who proved the analogous statement for torsion sheaves in \cite[Th\'eor\`emes de finitude, Th\'eor\`eme~2.13]{SGA4.5}, and whose proof for the case of $\Qlbar$-coefficients is given in \cite[Corollary~B.4]{BravGait}.

\subsubsection{}\label{sssec:spreadperv}
Let $K_0$ be a perverse sheaf on $A_0 = \cA_\eta$, and assume there exists an object $\cK \in \uD_{\cstr}(\cA,\Qlbar)$ with $j_\eta^*(\cK) \cong K_0$. Then there is a dense open subset $U \subset S$ such that the restriction of~$\cK$ to~$\cA_U$ is a relatively perverse sheaf $\cK_U \in \Perv^{\ULA}(\cA_U/U)$ on~$\cA_U$ over~$U$. (By~\ref{sssec:jeta*}, $\cK_U$ is then uniquely determined by~$K_0$, up to isomorphism.) To see this, we may assume, possibly after shrinking~$S$ and using~\ref{sssec:genericULA}, that $\cK \in \uD^{\ULA}(\cA,\Qlbar)$. Then the $\cP^i = {}^{\mathsf{p}/S}H^i(\cK)$ are in~$\uD^{\ULA}(\cA/S)$ and only finitely many of them are nonzero. For $i\neq 0$, the restriction of $\cP^i$ to~$\cA_\eta$ is zero. Since the (cohomology sheaves of the)~$\cP^i$ are constructible, it follows that $\cP^i$ (for $i \neq 0$) is supported on a constructible subset $C_i \subset \cA$ that does not meet the generic fibre~$\cA_\eta$. Because the morphism $\pi \colon \cA \to S$ is of finite type, a theorem of Chevalley (\cite[Corollaire~1.8.5]{EGAIV}) gives that $\pi(C_i) \subset S$ is constructible and does not contain the generic point~$\eta$. It follows that there is a dense open subset $U \subset S$ that is disjoint from $\bigcup_{i\neq 0} \pi(C_i)$. This~$U$ does the job, since ${}^{\mathsf{p}/S}H^i(\cK|_{\cA_U}) = 0$ for all $i \neq 0$, so $\cK|_{\cA_U} = {}^{\mathsf{p}/S}H^0(\cK|_{\cA_U})$ is perverse relative to~$U$.

\subsubsection{}\label{sssec:chiConstant}
If $\cK$ is an object of $\uD^{\ULA}(\cA/S)$ then, for every specialisation $u \colon \bar{t} \to S(\bar{s})$ of geometric points of~$S$ (here $S(\bar{s})$ denotes the strict henselisation of~$S$ at~$\bar{s}$), the cospecialisation map $\mathrm{cosp}(u) \colon R\Gamma(\cA_{\bar t},\cK|_{\cA_{\bar t}}) \to R\Gamma(\cA_{\bar s},\cK|_{\cA_{\bar s}})$ is an isomorphism; see Illusie's appendix to   \cite[Th\'eor\`emes de finitude]{SGA4.5}.  (This uses that $\cA \to S$ is proper.) In particular, the function $\bar s \mapsto \chi(\cK|_{\cA_{\bar s}})$ is constant.

\subsection{}\label{subsec:char0orp}
In the next results, we consider an abelian variety~$A$ over a field~$k$. We shall assume that either $\charact(k)=0$  or that $\charact(k)= p > 0$ and that we work in the following setting:
\begin{enumerate}[leftmargin=9ex,itemindent=0pt,label=\textup{(\thesubsection.\arabic*)}]
\item\label{charp} there is a subring $R \subset k$ which is finitely generated over~$\bbF_p$ and an abelian scheme~$\cA$ over $S = \Spec(R)$ such that $A = \cA \otimes_R k$. In this situation, let $k_0 \subset k$ be the fraction field of~$R$, and let $A_0 = \cA \otimes_R k_0$ be the generic fibre of~$\cA/S$.
\end{enumerate}
When we consider perverse sheaves on~$A$ over a field of characteristic~$p$, we shall need the assumption that these can be spread out to perverse sheaves on some model~$\cA$ over~$S$ as above. It will in fact be sufficient to assume that we can find $\cA \to S$ as in~\ref{charp} such that the perverse sheaf, say~$K$, can be obtained as the pullback of an object $\cK \in \uD_{\cstr}(\cA,\Qlbar)$. If this is true then it follows from~\ref{sssec:spreadperv} that, after replacing~$S$ by a nonempty open subset, we even have $\cK \in \Perv^{\ULA}(\cA/S)$. In this situation, we denote by $K_0$ the pullback of~$\cK$ to the generic fibre~$A_0$ of~$\cA \to S$ (in agreement with the convention introduced in Remark~\ref{rem:Notation0}).

\begin{proposition}\label{prop:chiP=0}
Let $A$ be an abelian variety over a field~$k$ and let $K$ and~$L$ be perverse sheaves on~$A$. Assume that either $\charact(k) = 0$, or that we are in the setting of~\ref{charp} and that $K$ and~$L$ are the pullbacks of constructible complexes $\cK$ and~$\cL$ on~$\cA$. If $P$ is a subquotient of ${}^{\mathsf{p}}H^i(K * L)$ with $i\neq 0$, we have $\chi(P) = 0$.
\end{proposition}

\begin{proof}
By Proposition~\ref{prop:chiPnonneg}, it suffices to prove $\chi\bigl({}^{\mathsf{p}}H^i(K * L)\bigr) = 0$ if $i\neq 0$. 

First assume $\charact(k) = 0$. By \cite[Lemma~A.1]{JavanEtAl}, it suffices to prove the assertion for $k=\bbC$. In this case, the assertion follows from the results of Kr\"amer and Weissauer in~\cite{KW}. (See also \cite[Lemma~2.2.(5)]{LawrSawin} and the arguments below.)

Assume next we are in the situation of \ref{charp} with $k_0 = R$ a finite field, so that we have a model~$A_0$ of~$A$ over~$k_0$, and $K$ and~$L$ are the pullbacks of perverse sheaves $K_0$ and~$L_0$ on~$A_0$. Suppose $L_0$ is semisimple. If $K_0$ is semisimple too, the assertion follows from \cite[Corollaries 6.3 and~6.4, Theorem~9.1]{KW}. In general, let $K_0^\prime = \soc_1(K_0) \subset K_0$ and $K_0^\pprime = K_0/K_0^\prime$. Because the functor ${}^{\mathsf{p}}H^0$ is cohomological (see \cite[Th\'eor\`eme~1.3.6]{BBD}), this gives an exact sequence of perverse sheaves
\[
\cdots \tto {}^{\mathsf{p}}H^i(K_0^\prime * L_0) \xrightarrow{~\alpha~} {}^{\mathsf{p}}H^i(K_0 * L_0) \xrightarrow{~\beta~} {}^{\mathsf{p}}H^i(K_0^\pprime * L_0) \tto {}^{\mathsf{p}}H^{i+1}(K_0^\prime * L_0) \tto \cdots\; .
\]
Arguing by induction on the socle length of~$K_0$, we may assume that $\chi\bigl({}^{\mathsf{p}}H^i(K_0^\prime * L_0)\bigr) = 0$ and $\chi\bigl({}^{\mathsf{p}}H^i(K_0^\pprime * L_0)\bigr) = 0$. It then follows from Proposition~\ref{prop:chiPnonneg} that also $\chi\bigl(\Image(\alpha)\bigr) = 0$ and $\chi\bigl(\Ker(\beta)\bigr) = 0$. Hence $\chi\bigl({}^{\mathsf{p}}H^i(K_0 * L_0)\bigr) = 0$. This gives the result for semisimple~$L_0$ and arbitrary~$K_0$. The general case (for $k_0$ finite) is handled in the same way, writing $L_0$ as an iterated extension of semisimple objects. 

Finally, consider the situation of \ref{charp}, where now $k_0$ is a finitely generated field extension of~$\bbF_p$. As explained in~\ref{subsec:char0orp}, we may then assume that $K_0$ and~$L_0$ are the pullbacks to the generic fibre of relatively perverse sheaves $\cK$, $\cL \in \Perv^{\ULA}(\cA/S)$. For a geometric point~$\bar{s}$ above a closed point $s \in S$, we then have $\chi\bigl({}^{\mathsf{p}}H^i(K * L)\bigr) = \chi\bigl({}^{\mathsf{p}}H^i(\cK_{\bar{s}} * \cL_{\bar{s}})\bigr)$. This reduces the problem to the case of a finite base field, which we have already dealt with.
\end{proof}

\subsection{}\label{subsec:Stab}
Let $P$ be a semisimple perverse sheaf on an abelian variety~$A$ over an algebraically closed field~$k$. As shown in \cite[Lemma~21]{WeissAlmostConn}, there is a reduced closed subgroup scheme $\Stab(P) \subset A$ with the property that
\[
\Stab(P)\bigl(k\bigr) = \bigl\{a\in A(k) \bigm| t_a^*(P) \cong P\bigr\}\; .
\]
Since this result is unpublished (as far as we are aware), let us sketch Weissauer's proof in the case where~$P$ is isotypical, which is the only case we need. Since $\Stab(P^{\oplus N}) = \Stab(P)$, we may assume that $P$ is simple. Write $S = \bigl\{a\in A(k) \bigm| t_a^*(P) \cong P\bigr\}$. The main point is to show that $S$ is constructible; because $S \subset A(k)$ is a subgroup, this implies that $S$ is closed in~$A(k)$, which gives the assertion.

To show that $S$ is constructible, we need two basic facts:
\begin{enumerate}
\item For arbitrary $K$, $L \in \uD_{\cstr}(A,\Qlbar)$, we have the relation $R\cHom\bigl(K,\bbD(L)\bigr) \cong \bbD(K \otimes^{\mathrm{L}} L)$, and this implies that $\Hom\bigl(K,\bbD(L)\bigr) \cong H^0\bigl(A,R\cHom(K,\bbD(L))\bigr) \cong H^0(A,K \otimes^{\mathrm{L}} L)^\vee$. 

\item For $a \in A(k)$, we have $\cH^0(K*L)_a \cong H^0\bigl(X,t_a^*(K) \otimes^{\mathrm{L}} (-\id_A)^*L\bigr)$.
\end{enumerate} 
Because~$P$ is a simple perverse sheaf, this gives
\[
a \in S \iff \Hom\bigl(t_a^*(P),P\bigr) \neq 0 \iff H^0\bigl(A,t_a^*(P) \otimes^{\mathrm{L}} \bbD(P)\bigr) \neq 0
\iff a \in \Supp\bigl(\cH^0(P * P^\vee)\bigr)\, ,
\]
where $P^\vee = (-\id_A)^*\bbD(P) \in \uD_{\cstr}(A,\Qlbar)$ is the dual of~$P$ for the convolution product. This shows that $S = \Supp\bigl(\cH^0(P * P^\vee)\bigr)$, which is constructible.

\begin{proposition}\label{prop:dimStab>0}
Let $A$ be an abelian variety over an algebraically closed field~$k$ and let $P$ be an isotypical perverse sheaf on~$A$. Assume that either $\charact(k) = 0$, or that we are in the setting of~\ref{charp} and that $P$ is the pullback of a constructible complex $\cP \in \uD_{\cstr}(\cA,\Qlbar)$. If $\chi(P) = 0$, then $\dim\bigl(\Stab(P)\bigr) > 0$.
\end{proposition}

\begin{proof}
If $\charact(k) = 0$, we can reduce to the case $k=\bbC$ by using \cite[Lemma~A.1]{JavanEtAl}  and, in that case, the assertion follows from \cite[Proposition~10.1]{KW}. 

If $\charact(k) = p$, the assertion follows from the main theorem in~\cite{WeissFF}. However, if we are in the situation of \ref{charp} with $k_0$ a finitely generated field extension of~$\bbF_p$, the required `spreading out' arguments are only outlined in~\cite{WeissFF}. We here explain this using the techniques that we have discussed in \ref{ssec:spreadingout}--\ref{ssec:ULAprops}. Assume $\Stab(P)$ is finite. Choose an integer~$n$ such that $\Stab(P) \subset A[n]$  and choose a prime number~$q$ prime to $pn$. After replacing~$k_0$ by a finite extension, we may assume that all $q$-torsion points of~$A$ are rational over~$k_0$. In this step, replace~$S$ by its normalization in the new field~$k_0$, and then shrink it so that the pullback of~$\cP$ to the new~$S$ is in~$\Perv^{\ULA}(\cA/S)$; see~\ref{subsec:char0orp}. By~\ref{sssec:genericULA}, we may further assume, possibly after again shrinking~$S$, that all $R\cHom\bigl(t_a^*(\cP),t_b^*(\cP)\bigr)$ with $a$, $b \in A_0[q]\bigl(k_0\bigr) = \cA[q]\bigl(S\bigr)$ are universally locally acyclic. By~\ref{sssec:chiConstant}, it follows from this that, for every $\bar{s} \in S(\Fpbar)$, we have
\begin{equation}\label{eq:Hom=Hom}
\Hom_{\Perv(A)}\bigl(t_a^*(P),t_b^*(P)\bigr) \isomarrow \Hom_{\Perv(\cA_{\bar{s}})}\bigl(t_a^*(\cP_{\bar{s}}),t_b^*(\cP_{\bar{s}})\bigr)\; .
\end{equation}
In particular, $\End(P) \isomarrow \End(\cP_{\bar{s}})$ and, since $P$ is isotypical, $\cP_{\bar{s}}$ is isotypical as well. Because the proposition is true for~$\cP_{\bar{s}}$ by the main result of~\cite{WeissFF}, its stabiliser $\Stab(\cP_{\bar{s}})$ is a subgroup scheme of~$A$ of positive dimension. If $a$ is a point of order~$q$ on~$\Stab(\cP_{\bar{s}})$, it follows from~\eqref{eq:Hom=Hom} that $a \in \Stab(P)$; but this contradicts our choice of~$q$. This shows that $\Stab(P)$ cannot be finite.
\end{proof}

\begin{proposition}\label{prop:K*Lperv}
Let $A$ be an absolutely simple abelian variety of dimension~$g$ over a field~$k$ and let $K$ and~$L$ be perverse $\Qlbar$-sheaves on~$A$. Assume that either $\charact(k) = 0$, or that we are in the setting of~\ref{charp} and that $K$ and~$L$ are the pullbacks of constructible complexes $\cK$ and~$\cL$ on~$\cA$. Assume further that there exist closed subvarieties $X$ and~$Y$ of~$A$ such that $\Supp(K) \subset X$ and $\Supp(L) \subset Y$, and such that $\dim(X) + \dim(Y) \leq g$. Then $K * L$ is again a perverse sheaf.
\end{proposition}

\begin{proof}
In order to treat the cases $\charact(k)=0$ and $\charact(k) = p$ simultaneously, we change (for the duration of the proof) the notation as follows: if $\charact(k) = 0$, we write $k_0$ and~$A_0$ instead of~$k$ and~$A$, and we write $K_0$ and~$L_0$ instead of $K$ and~$L$; then, we define~$k$ to be an algebraic closure of~$k_0$, and we follow the notational convention of Remark~\ref{rem:Notation0}. (Note that if $K_0 * L_0$ is perverse after base change to $k = \ol{k_0}$, it is perverse over~$k_0$.)

By Lemma~\ref{lem:SSPervFq}\ref{Isotyp/k1}, we may, after replacing~$k_0$ by a finite extension, assume that every simple subquotient~$P_0$ of $K_0*L_0$ (that is, every composition factor of $K_0 * L_0$) has the property that its pullback~$P$ to~$A$ is isotypical. 

We assume that $K_0 * L_0$ is not perverse; our goal is to derive a contradiction. Let $i$ be the largest integer such that ${}^{\mathsf{p}}H^i(K_0 * L_0) \neq 0$. Without loss of generality, we may assume $i>0$; otherwise, replace $K_0$ and~$L_0$ by their Verdier duals and use~\eqref{eq:D(K*L)}. Let $P_0$ be a simple quotient of ${}^{\mathsf{p}}H^i(K_0 * L_0)$ in~$\Perv(A_0)$. If $\charact(k) = p$, we may assume, as explained in~\ref{subsec:char0orp} and possibly after shrinking~$S$, that $\cK$ and~$\cL$ are in~$\Perv^{\ULA}(\cA/S)$, and then it follows from~\ref{sssec:jeta*} that also~$P_0$ is the pullback to the generic fibre~$A_0$ of a relatively perverse sheaf. By Propositions~\ref{prop:chiP=0} and~\ref{prop:dimStab>0}, we have $\dim\bigl(\Stab(P)\bigr) > 0$. Because $A$ is simple and the support of~$P$ is contained in $X+Y$, it follows that $X + Y = \Stab(P) = A$, which implies $\dim(X) + \dim(Y) = g$. Moreover, because $\pi_1(A,0)$ is abelian and $P$ is isotypical, there exists a continuous character $\phi \colon \pi_1(A,0) \to \ol\bbQ^\times_\ell$ such that $P$ is a power of~$\bbL_\phi[g]$. (Use~\cite[Th\'eor\`eme~4.3.1]{BBD}.) Let $\psi = \phi^{-1}$. By how~$i$ was chosen and the relation \eqref{eq:Kpsi*Lpsi}, we have a distinguished triangle of the form $\cF \tto K_\psi * L_\psi \tto \Qlbar[g-i] \tto$ with $\cF \in {}^{\mathsf{p}}\uD^{\leq i}(A,\Qlbar)$ and $i>0$. This gives a long exact sequence
\[
\cdots \tto H^{g+i}(A,K_\psi * L_\psi) \tto H^{2g}(A,\Qlbar) \tto H^{g+i+1}(A,\cF) \tto \cdots\; .
\]
Because $\cF \in {}^{\mathsf{p}}\uD^{\leq i}(A,\Qlbar)$, we have $H^{g+i+1}(A,\cF) = 0$. Also, $H^\bullet(A,K_\psi * L_\psi) \cong H^\bullet(A,K_\psi) \otimes H^\bullet(A,L_\psi)$, which vanishes in degrees above $\dim(X) + \dim(Y) = g$. This gives a contradiction with  the fact that $H^{2g}(A,\Qlbar) \neq 0$.
\end{proof}

\section{The addition map on an absolutely simple abelian variety}\label{sec:AddSimpleAV}

In this section, we prove the main theorems stated in the introduction.

\begin{theorem}\label{thm:semismall}
Let $A$ be an absolutely simple abelian variety of dimension~$g$ over a field~$k$. Let $X$ and~$Y$ be irreducible closed subvarieties of~$A$ and define $Z \coloneqq X+Y \subset A$. If $\dim(X) + \dim(Y) \leq g$, the addition morphism $\sigma\colon X\times Y \to Z$ is semismall.
\end{theorem}

After some preparations, the proof of the theorem will be finished in Section~\ref{subsec:pfmaincalc}. In Corollary~\ref{cor:ssmall}, we shall extend the result to the case of more than two subvarieties.

\begin{remark}\label{rem:ThmGeometric}
Theorem~\ref{thm:semismall} is of a geometric nature. Suppose $k \subset k^\prime$ is a field extension and define $A^\prime \coloneqq A \otimes_k k^\prime$. Let $X^\prime \subset A^\prime$ be the reduced underlying subscheme of $X \otimes_k k^\prime$, and  define~$Y^\prime$ similarly. Then $X \times Y \to Z$ is semismall if and only if $X^\prime \times Y^\prime \to X^\prime + Y^\prime$ is semismall, and this can be checked on each irreducible component of $X^\prime \times Y^\prime$ separately. In the proof, we may therefore assume that $k$ is algebraically closed.
\end{remark}

\subsection{}\label{subsec:ideaproof}
To explain the idea of the proof, suppose that $k$ is algebraically closed and consider the special case where $X$ and~$Y$ are irreducible and \emph{smooth} over~$k$, of dimensions $d=\dim(X)$ and $e=\dim(Y)$. Their intersection complexes are then simply $i_{X,*}\Qlbar[d]$ and~$i_{Y,*}\Qlbar[e]$, where $i_X \colon X\hookrightarrow A$ and $i_Y \colon Y\hookrightarrow A$ are the inclusions. Now we can argue as follows.
\begin{itemize}
\item If $z \in Z(k)$ and $F_z \subset X\times Y$ is the fibre of~$\sigma$ over~$z$, the stalk of the sheaf $\cH^m(\IC_X * \IC_Y)$ at~$z$ is $H^{m+d+e}(F_z,\Qlbar)$. (Here things greatly simplify because of our assumption that $X$ and~$Y$ are smooth over~$k$.)

\item If $\dim(F_z) = n$, then $H^{2n}(F_z,\Qlbar) \neq 0$ (see Lemma~\ref{lem:topcoh} below).

\item Proposition~\ref{prop:K*Lperv} gives that $\IC_X * \IC_Y$ is a perverse sheaf on~$A$. This implies that the locus of points $z\in Z$ such that $\cH^m(\IC_X * \IC_Y)_z \neq 0$ has dimension at most~$-m$. Taking $m = 2n-d-e$, we find that the locus of points $z\in Z$ with $\dim(F_z) = n$ has dimension at most $d+e-2n$, which is exactly the definition of semismallness.
\end{itemize}

\noindent
(This type of argument is well-known to experts; cf.\ \cite[Proposition~4.2.1]{dCM}.) For the proof in the general case, we start with a lemma that is surely well known but for which we could not find a good reference.

\begin{lemma}\label{lem:topcoh}
Let $V$ be a nonempty scheme of finite type over an algebraically closed field, with $n\coloneqq \dim(V)  \ge 0$. Then $H^{2n}_c(V,\Qlbar) \neq 0$. 
\end{lemma}

\begin{proof}
Since \'etale cohomology only depends on the underlying reduced scheme, we may assume $V$ is reduced. We have an exact sequence
\[
\cdots \tto H^{m-1}_c(V^{\sing},\Qlbar) \tto H^m_c(V^{\reg},\Qlbar) \tto H^m_c(V,\Qlbar) \tto H^m_c(V^{\sing},\Qlbar) \tto \cdots
\]  
and $H^m_c(V^{\sing},\Qlbar) = 0$ for $m\geq 2n-1$. Moreover, $H^{2n}_c(V^{\reg},\Qlbar)$ is dual to $H^0(V^{\reg},\Qlbar)$, which is nonzero since $V$ is nonempty and reduced.
\end{proof}

\subsection{}\label{subsec:k=CFpbar}
With notation and assumptions as in Theorem~\ref{thm:semismall}, let $d \coloneqq \dim(X)$ and $e \coloneqq \dim(Y)$. Choose a subfield $k_0 \subset k$ which is finitely generated over its prime field, such that $A$, $X$, and~$Y$ are obtained from $A_0$, $X_0$, and~$Y_0$ over~$k_0$ by extension of scalars. Choose an algebraic closure $k_0 \subset k$, redefine $A$ to be $A = A_0 \otimes_{k_0} k$, and redefine $X$ and~$Y$ to be $(X_0 \otimes_{k_0} k)_{\red}$ and $(Y_0 \otimes_{k_0} k)_{\red}$, respectively. If the theorem is valid for the $A$, $X$, and~$Y$ thus obtained, it is valid for the original $A$, $X$ and~$Y$. This reduces the problem to the situation where $k$ is an algebraic closure of a field~$k_0$ which is finitely generated over its prime field, and where $A$, $X$, and~$Y$ are obtained from $A_0$, $X_0$, and~$Y_0$ over~$k_0$ by extension of scalars. In the remainder of the proof, we assume that we are in this situation.

\subsection{}\label{subsec:goodopens}
Choose affine regular open subsets $j_X \colon U_X \hookrightarrow X$ and $j_Y \colon U_Y \hookrightarrow Y$ that are defined over~$k_0$. Let $i_X \colon \Delta_X \hookrightarrow X$ and $i_Y \colon \Delta_Y \hookrightarrow Y$ be the (reduced) closed complements. Let $U \coloneqq U_X \times U_Y$ and $\Delta \coloneqq (X\times Y)\smallsetminus U = (\Delta_X \times Y) \cup (X\times \Delta_Y)$, and write $j \colon U \hookrightarrow X\times Y$ and $i\colon \Delta \hookrightarrow X\times Y$ for the inclusion maps. 

We have $\IC_{X\times Y} = \IC_X \boxtimes\, \IC_Y$. Define
\[
K \coloneqq \IC_X|_{\Delta_X}[-1]\, ,\qquad L \coloneqq \IC_Y|_{\Delta_Y}[-1]\, ,\qquad J \coloneqq \IC_{X\times Y}|_\Delta[-1]\, .
\]
By \cite[Corollaire~4.1.12]{BBD}, these are perverse sheaves. Clearly, $\IC_X$, $\IC_Y$, $K$, and~$L$ are pullbacks of perverse sheaves on~$A_0$; however, in order to be able to apply Proposition~\ref{prop:K*Lperv}, we need to show that they can be extended to constructible complexes on some model $\cA \to S$ of~$A_0$ as in~\ref{subsec:char0orp}. First consider~$\IC_X$. By~\ref{sssec:jeta*}, it suffices to show that $j_{X,!}\Qlbar[d]$ can be extended to a relative perverse sheaf, because~$\IC_X$ is a quotient of it. Now consider a model $\cA \to S$ as in~\ref{subsec:char0orp} and an irreducible affine subscheme $\iota \colon \cV \hookrightarrow \cA$ which is smooth of relative dimension~$d$ over~$S$ and such that $\cV\otimes k = U_X$. Then $\iota_!\Qlbar[d]$ is a relatively perverse sheaf (use base change for $!$-pushforwards, see \cite[Lemma~6.7.10]{BS}, and \cite[Corollaire~4.1.3]{BBD}), and its restriction to~$A$ is~$j_{X,!}\Qlbar[d]$. The same argument of course works for~$\IC_Y$. It is then immediate that $K$ and~$L$, which up to a shift are simply pullbacks of~$\IC_X$ and~$\IC_Y$, can be extended to constructible objects on some model. This justifies why we can apply Proposition~\ref{prop:K*Lperv}.

By adjunction, we have morphisms 
\[
K \boxtimes\, \IC_Y \xleftarrow{~\alpha_X~} J \xrightarrow{~\alpha_Y~} \IC_X\boxtimes L
\]
and we define $\alpha \coloneqq (\alpha_X,\alpha_Y) \colon J \to (K \boxtimes\, \IC_Y) \oplus (\IC_X\boxtimes L)$. Similarly, we have morphisms $\beta_X \colon \IC_X \to K[1]$ and $\beta_Y \colon \IC_Y \to L[1]$, and we define
\[
\beta \coloneqq (\id \boxtimes \beta_Y) - (\beta_X \boxtimes \id) \colon (K \boxtimes\, \IC_Y) \oplus (\IC_X\boxtimes L) \tto K\boxtimes L[1]\, .
\]

\begin{lemma}
We have a distinguished triangle
\begin{equation}\label{eq:JKLtriang}
J \xrightarrow{~\alpha~} (K \boxtimes\, \IC_Y) \oplus (\IC_X\boxtimes L) \xrightarrow{~\beta~} K\boxtimes L[1] \tto\; .
\end{equation}
\end{lemma}

\begin{proof}
Let
\[
(\Delta_X \times Y) \xhookrightarrow{~\iota_X~} \Delta \xhookleftarrow{~\iota_Y~} (X\times\Delta_Y) \, ,\qquad 
(\Delta_X \times \Delta_Y) \xhookrightarrow{~\theta~} \Delta
\]
be the inclusion maps. If $\cF$ is any $\Qlbar$-sheaf on~$\Delta_\proet$, we have a short exact sequence
\[
0 \tto \cF \xrightarrow{~\alpha~} (\iota_{X})_*\iota_X^* \cF \oplus  (\iota_{Y})_*\iota_Y^* \cF \xrightarrow{~\beta~} \theta_*\theta^* \cF \tto 0\, ,
\]
where $\alpha$ is the sum of the adjunction maps and $\beta$ is the difference of the adjunction maps (exactly as in the construction above). Passing to the derived category  gives the lemma.
\end{proof}

\subsection{}\label{subsec:pfmaincalc}
Keeping the notation and assumptions as above, let $\sigma \colon X\times Y \to Z$ be the addition map. For $z \in Z(k)$, let $F_z \coloneqq \sigma^{-1}(z)$ be the fibre of~$\sigma$ over~$z$. Let $V_z \coloneqq F_z \cap U$ and $D_z \coloneqq F_z \cap \Delta$.

By proper base change, the stalk of $\cH^m(\IC_X * \IC_Y)$ at~$z$ is isomorphic to $H^m(F_z,\IC_{X\times Y}|_{F_z})$. The short exact sequence of perverse sheaves 
\[
0 \tto i_*J \tto j_!\, \Qbar_{\ell,U}[d+e] \tto \IC_{X\times Y} \tto 0
\]
(see \cite[(4.1.12.3)]{BBD}) gives rise to a long exact cohomology sequence
\begin{multline}\label{eq:cohseq1}
\cdots \tto H^m\bigl(D_z,J|_{D_z}\bigr) \tto H^{m+d+e}_c(V_z,\Qlbar)\tto\\ 
H^m\bigl(F_z,\IC_{X\times Y}|_{F_z}\bigr)\tto H^{m+1}\bigl(D_z,J|_{D_z}\bigr) \tto \cdots\; .
\end{multline}
(We use the base change theorem for $!$-pushforwards; see \cite[Lemma~6.7.10]{BS}.) Moreover, the triangle~\eqref{eq:JKLtriang} gives a long exact sequence
\begin{multline}
\cdots \tto H^m\bigl(D_z,K\boxtimes L|_{D_z}\bigr) \tto H^m\bigl(D_z,J|_{D_z}\bigr) 
\tto \\
H^m\bigl(D_z,K\boxtimes\, \IC_Y|_{D_z}\bigr) \oplus H^m\bigl(D_z,\IC_X\boxtimes L|_{D_z}\bigr)
\tto H^{m+1}\bigl(D_z,K\boxtimes L|_{D_z}\bigr) \tto \cdots\; .
\end{multline}
Because $H^m\bigl(D_z,K\boxtimes L|_{D_z}\bigr)$ is the stalk at~$z$ of $\cH^m(K*L)$, the perversity of $K*L$ (Proposition~\ref{prop:K*Lperv}) implies that the locus of points~$z$ with $H^m\bigl(D_z,K\boxtimes L|_{D_z}\bigr) \neq 0$ has dimension at most~$-m$. Exactly the same holds for the terms $H^m\bigl(D_z,K\boxtimes\, \IC_Y|_{D_z}\bigr)$ and $H^m\bigl(D_z,\IC_X\boxtimes L|_{D_z}\bigr)$.

For $n\geq 0$, define $Z_n^\circ = \bigl\{z\in Z\bigm| \dim(F_z) = n\bigr\}$, which is a locally closed subset of~$Z$. Taking $m = 2n-d-e$ in the previous calculations, we find that there is a locus $Z^\prime_n \subset Z_n^\circ$ of dimension at most $d+e-2n$, such that
\begin{equation}\label{eq:HDvanish}
H^{2n-d-e}\bigl(D_z,J|_{D_z}\bigr) = H^{2n-d-e+1}\bigl(D_z,J|_{D_z}\bigr) = 0
\end{equation}
for all $z \in Z_n^\circ\smallsetminus Z^\prime_n$.

Let $Z_n^\pprime = \bigl\{z\in Z_n^\circ\bigm| \dim(V_z) = n\bigr\}$. For $z \in Z_n^\pprime$, we have $H^{2n}_c(V_z,\Qlbar) \neq 0$ by Lemma~\ref{lem:topcoh}. By the exact sequence~\eqref{eq:cohseq1}, it follows that $H^{2n-d-e}\bigl(F_z,\IC_{X\times Y}|_{F_z}\bigr) \neq 0$ for all $z \in Z_n^\pprime\smallsetminus Z_n^\prime$. By perversity of $\IC_X * \IC_Y$ (again using Proposition~\ref{prop:K*Lperv}), it follows that $\dim(Z_n^\pprime\smallsetminus Z_n^\prime) \leq d+e-2n$. 

Finally, for $z \in Z_n^\circ\smallsetminus (Z_n^\prime \cup Z_n^\pprime)$, we have $\dim(D_z) = n$. By induction, we may assume that the addition maps $\Delta_X \times Y \to \Delta_X + Y$ and $X\times \Delta_Y \to X+\Delta_Y$ are semismall, so the locus of points~$z$ where $\dim(D_z) = n$ has dimension at most $d+e-1-2n$. 

In total, this gives $\dim(Z_n^\circ) \leq d+e-2n$, hence the addition map $X\times Y \to Z$ is semismall. This completes the proof of Theorem~\ref{thm:semismall}. \qed

\begin{corollary}\label{cor:dimZ}
Let $A$ be an absolutely simple abelian variety of dimension~$g$ over a field. Let $X_1,\ldots, X_r$ be closed subvarieties of~$A$ and define $Z \coloneqq X_1 + \cdots + X_r$. We have
\[
\dim(Z) = \min\biggl\{g,\sum_{i=1}^r\, \dim(X_i)\biggr\}\, .
\] 
\end{corollary}

\begin{proof}
It suffices to prove this for $r=2$. Let $X_1^\prime \subset X_1$ and $X_2^\prime \subset X_2$ be irreducible closed subvarieties such that $\dim(X_1^\prime)+\dim(X_2^\prime)= \min\{g,\dim(X_1)+\dim(X_2)\}$. Theorem~\ref{thm:semismall}, applied to $X_1^\prime$, $X_2^\prime \subset A$, gives that the addition map $\sigma \colon X_1^\prime \times X_2^\prime \to X_1^\prime + X_2^\prime$ is semismall. In particular, $\sigma$ is generically finite and we find  
\[
\dim(X_1+ X_2) \ge \dim(X_1^\prime + X_2^\prime)  = \dim(X_1^\prime) +\dim(X_2^\prime)= \min\{g,\dim(X_1)+\dim(X_2)\}\, .
\]
The reverse inequality is clear.
\end{proof}

\begin{corollary}\label{cor:codim}
Let $A$ be an absolutely simple abelian variety of dimension~$g$ over a field. If $X_1, \ldots, X_r$ are closed subvarieties of~$A$ such that $\sum_{i=1}^r\, \codim_A(X_i)\le g$, then $\bigcap_{i=1}^r X_i\neq \varnothing$.
\end{corollary}

\begin{proof}
We argue by induction on~$r$. For $r=1$, there is nothing to prove, and the induction step reduces to the case $r=2$. For $r=2$, suppose  $\codim_A(X_1) + \codim_A(X_2) \leq g$ (by definition, $\codim_A(X_i)\coloneqq g-\dim(X_i)$). By Corollary~\ref{cor:dimZ}, the difference map $X_1 \times X_2 \to A$ is then surjective. Since $X_1\cap X_2$ is isomorphic to its fibre over~$0$, we obtain $X_1 \cap X_2 \neq \varnothing$. 
\end{proof}

We now extend Theorem~\ref{thm:semismall} to the case of more than two subvarieties.

\begin{corollary}\label{cor:ssmall}
Let $A$ be an absolutely simple abelian variety of dimension~$g$ over a field~$k$, let $X_1,\ldots,X_r$ be irreducible closed subvarieties of~$A$, and define $Z \coloneqq X_1 + \cdots + X_r$. If $\sum_{i=1}^r\, \dim(X_i) \leq g$, the addition map $\sigma \colon X_1 \times \cdots \times X_r \to Z$ is semismall.
\end{corollary}

\begin{proof}
Without loss of generality, we may assume that $k$ is algebraically closed. We argue by induction on~$r$. The case $r=1$ is trivial and the case $r=2$ is Theorem~\ref{thm:semismall}. Assume then that $r\geq 3$ and that the assertion is true whenever we consider at most $r-1$ subvarieties~$X_i$. In particular, if we set $Y \coloneqq X_1 + \cdots + X_{r-1}$, the addition maps
\[
X_1 \times \cdots \times X_{r-1} \xrightarrow{~\alpha~} Y\, ,\qquad
Y \times X_r \xrightarrow{~\beta~} Z
\]
are semismall.

For $n \geq 0$, let $Z_n \subset Z$ be the closed subset of points $z\in Z$ such that $\dim \sigma^{-1}(z) \geq n$, and define $Z_n^\circ \coloneqq Z_n\smallsetminus Z_{n+1}$. Our goal is to prove that $Z_n$ has dimension at most $\dim(Z) - 2n$. (By Corollary~\ref{cor:dimZ} we have $\dim(Z) = \sum_{i=1}^r\, \dim(X_i)$.) Let $\Gamma^\circ \subset Z_n^\circ$ be an irreducible component and let $\Gamma \subset Z_n$ be its Zariski closure. Let $\gamma \in \Gamma$ be the generic point. Let $T_\gamma \subset \sigma^{-1}(\gamma)$ be an irreducible component of dimension~$n$, let $T \subset X_1 \times \cdots \times X_r$ be the Zariski closure of~$T_\gamma$, and let $V \coloneqq (\alpha \times \id)\bigl(T\bigr)$ be the image of~$T$ in $Y\times X_r$. By construction, $\sigma(T) = \beta(V) = \Gamma$. Also, the fibre of $V \to \Gamma$ over~$\gamma$ is irreducible; let $e$ be its dimension. 

Let $W \subset Y$ be the image of $V \subset Y\times X_r$ under the first projection map. Then $V \subset W \times X_r$, and hence $\Gamma \subset W+X_r$. Because the map $W\times X_r \to W+X_r$ is semismall, we have
\begin{equation}\label{eq:dimGam}
\dim(\Gamma) \leq \dim(W) + \dim(X_r) - 2e\; .
\end{equation}

By construction, the fibre of $\alpha \times \id \colon (X_1 \times \cdots \times X_{r-1}) \times X_r \to Y \times X_r$ over the generic point of~$V$ has dimension~$n-e$. It follows that the fibre of~$\alpha$ over the generic point of~$W$ also has dimension at least $n-e$. Because $\alpha$ is semismall, this implies   $\dim(W) \leq \dim(Y) - 2n + 2e$ and, together with~\eqref{eq:dimGam}, we obtain $\dim(\Gamma)\leq \dim(Y) + \dim(X_r) -2n = \dim(Z) - 2n$, which is what we wanted to prove.
\end{proof}

\begin{remark}\label{rem:martens}
An classical theorem of H.~Martens (\cite[Theorem~1]{Martens}) says that if $C$ is a smooth irreducible projective curve of genus $g$ and $d\le g$, we have
$$\dim(W^n_d(C))\le d-2n
$$
for any $n\ge 0$ (here, $W^n_d(C)$ is the set of degree $d$ line bundles $L$ on $C$ such that $h^0(C,L)> n$). 
This is equivalent to saying that
 the addition map $C^d\twoheadrightarrow W^0_d(C)\subset \Jac(C)$ is semismall. 
Note that Martens does not assume that $\Jac(C)$ is absolutely simple; when it is, his result is a particular case of Corollary~\ref{cor:ssmall}.

More generally, if $X\subset A$ is an irreducible curve that generates a (not necessarily simple) abelian variety~$A$, we can show that for $d\le \dim(A)$, the addition map $X^d\to A$ is semismall onto its image. The proof is elementary and follows the ideas used in the proof of~\cite[Proposition~3.3]{debfah}. 
\end{remark}

\begin{remark}\label{rem:simpleAV}
The following construction shows that in Corollary~\ref{cor:dimZ}, it does not suffice to assume that $A$ is simple (as opposed to absolutely simple).

Let $k$ be a field with separable algebraic closure $k \subset \ksep$. Let $k \subset L$ be a finite separable field extension and write $\Emb(L)$ for the set of $k$-homomorphisms $\sigma \colon L \to \ksep$. Let $\tilde{L} \subset \ksep$ be the compositum of all subfields $\sigma(L) \subset \ksep$, for $\sigma \in \Emb(L)$, which is a finite Galois extension of~$k$. In what follows, we view the elements $\sigma \in \Emb(L)$ as embeddings $L \hookrightarrow \tilde{L}$. The Galois group $\Gamma  = \Gal(\tilde{L}/k)$ naturally acts on~$\Emb(L)$ and this action is transitive.

If $B$ is an abelian variety over~$L$ and $\sigma \in \Emb(L)$, let $B_\sigma$ denote the abelian variety over~$\tilde{L}$ that is obtained by extension of scalars via~$\sigma$. If $A = \Res_{L/k}\, B$ is the abelian variety over~$k$ obtained by restriction of scalars, we have $A_{\tilde{L}} \cong \prod_{\sigma \in \Emb(L)}\; B_\sigma$. The natural action of~$\Gamma$ on~$A_{\tilde{L}}$ permutes the factors~$B_\sigma$.

If $I \subset \Emb(L)$ is a subset, define $T_I \subset A_{\tilde{L}} = \prod_{\sigma \in \Emb(L)}\; B_\sigma$ by 
\[
T_I \coloneqq \biggl\{(b_\sigma) \in \prod_{\sigma \in \Emb(L)}\; B_\sigma \biggm| \text{$b_\sigma = 0$ for all $\sigma \notin I$} \biggr\}\, .
\]
Clearly, $T_I$ is an abelian subvariety of~$A_{\tilde{L}}$ of dimension equal to $\# I \cdot \dim(B)$. If $\gamma \in \Gamma$, we have ${}^\gamma T_I = T_{\gamma(I)}$. Furthermore, for $I$, $J \subset \Emb(L)$, we have $T_I + T_J = T_{I\cup J}$.

Let $\cI$ be a $\Gamma$-orbit of subsets of~$\Emb(L)$. By Galois descent, there is a closed subscheme $X_\cI \subset A$ over~$k$ such that
\[
X_\cI \otimes_k \tilde{L} = \bigcup_{I\in \cI}\; T_I\, .
\]

We now make several choices that will give rise to an example of the desired kind. To begin with, we choose a field extension $k \subset L$ as above and two $\Gamma$-orbits $\cI$, $\cJ$ of subsets of~$\Emb(L)$ such that:
\begin{itemize}
\item if $I \in \cI$ and $J \in \cJ$, then $\# I + \# J = [L:k]$;
\item $I \cap J \neq \varnothing$ for all $I \in \cI$ and $J \in \cJ$.
\end{itemize}
(Note that $[L:k] = \# \Emb(L)$, so if the first condition is satisfied, the second condition is equivalent to $I \cup J \neq \Emb(L)$ for all $I \in \cI$ and $J \in \cJ$.) Next, we choose an absolutely simple abelian variety $B$ over~$L$ with the property that $B_\sigma$ and~$B_\tau$ are not isogenous whenever $\sigma \neq \tau$ in $\Emb(L)$. This condition implies that the abelian variety $A = \Res_{L/k}\, B$ is simple over~$k$. With notation as above, the closed subschemes $X_\cI$, $X_\cJ$ of~$A$ then satisfy $\dim(X_\cI) + \dim(X_\cJ) = \dim(A)$, whereas $X_\cI + X_\cJ \neq A$.

Finally, we note that it is possible to make choices as above. For a concrete example, take $k = \bbQ$ and $L = \tilde{L} = \bbQ(\zeta_7)$ with $\zeta_7 = \exp(2\pi i/7)$. Number the complex embeddings of~$L$ in such a way that $\Gamma = \Gal(L/\bbQ)$ is generated by the cyclic permutation $(1\; 2\; 3\; 4\; 5\; 6)$, and take for  $\cI$  the $\Gamma$-orbit of $I = \{1,2,3\}$ and  for $\cJ$ the $\Gamma$-orbit of $\{1,3,5\}$. If $B$ is an elliptic curve over~$L$ that is not isogenous to any of its Galois conjugates, then $A = \Res_{L/\bbQ}\, B$ is a $6$-dimensional simple abelian variety over~$\bbQ$ and $X_\cI$, $X_\cJ \subset A$ are $3$-dimensional closed subvarieties with $\dim(X_\cI + X_\cJ) = 5$.
\end{remark}

\section{The addition map on arbitrary abelian varieties}\label{sec:AddArbitraryAV}

It is clear that Corollary~\ref{cor:dimZ}  does not extend as is to nonsimple abelian varieties. However, we shall show that the result remains valid under an additional hypothesis. We use the following notion  introduced by Ran in~\cite{Ran}. 

\begin{definition}\label{def:gndeg}
Let $A$ be an abelian variety over a field~$k$. If $k=\kbar$, an irreducible closed subvariety~$X$ of~$A$ is said to be geometrically nondegenerate if for every quotient morphism $q \colon A \to A/B$, with $B\subset A$ an abelian subvariety, either $q(X)= A/B$ or $\dim\bigl(q(X)\bigr) = \dim(X)$. Equivalently, $X$ is geometrically nondegenerate if
\begin{equation}\label{eq:geomnondeg}
\dim(X+B)= \min\bigl\{\dim(A),\dim(X)+\dim(B)\bigr\}
\end{equation}
for every abelian subvariety $B\subset A$ (compare with Corollary~\ref{cor:X1Xr}).

For arbitrary~$k$, we say that a closed subvariety $X\subset A$ is geometrically nondegenerate if all irreducible components of~$X_{\kbar}$ are geometrically nondegenerate.
\end{definition}

Note the connection between \eqref{eq:geomnondeg} and Corollary~\ref{cor:X1Xr}.

\subsection{Examples and elementary properties.} \label{sec:gndeg}Assume $k=\kbar$.
\begin{enumerate}
\item If $A$ is simple, every  irreducible  closed subvariety of~$A$ is geometrically nondegenerate.

\item If $q\colon A \twoheadrightarrow A^\prime$ is a surjective homomorphism of abelian varieties and $X \subset A$ is a geometrically nondegenerate closed subvariety, its image $q(X) \subset A^\prime$ is again geometrically nondegenerate.

\item An irreducible curve $X \subset A$ is geometrically nondegenerate if and only if $X$ generates~$A$.

\item An irreducible hypersurface $X \subset A$ is geometrically nondegenerate if and only if $X$ is ample.

\item A dimensionally transverse intersection $X \subset A$ of ample hypersurfaces is geometrically nondegenerate. (We recall that an intersection $Y\cap Z$ of closed subschemes is dimensionally transverse if $\Tor_{>0}(\cO_Y,\cO_Z)=0$.) To see this, suppose we have a quotient $q\colon A\to A/B$ with $q(X)\subsetneq A/B$ and $\dim\bigl(q(X)\bigr) < \dim(X)$. Take $Z\subset A/B$   of dimension $\dim(A/B)-\dim\bigl(q(X)\bigr)-1$ (a nonnegative number) and disjoint from~$q(X)$. Then $\dim\bigl(q^{-1}(Z)\bigr)=\dim(Z)+\dim(B)=\dim(A)-\dim\bigl(q(X)\bigr)-1\ge \dim(A)-\dim(X)$, but $q^{-1}(Z)$ is disjoint from~$X$. Since $X$ is a dimensionally transverse intersection of ample hypersurfaces, this is impossible.
\end{enumerate}
\medskip

The main result of this section is the following.

\begin{theorem}\label{thm:X+Y=A}
Let $A$ be an abelian variety of dimension~$g$ over a field~$k$. Let $X$ and~$Y$ be closed subvarieties of~$A$ with $\dim(X) + \dim(Y) \geq g$ and assume $X$ is geometrically nondegenerate. Then $X+Y = A$.
\end{theorem}

\begin{proof}
By the same argument as in~\ref{subsec:k=CFpbar}, we may assume that either $k=\bbC$ or that $k$ is the algebraic closure of a field~$k_0$ which is finitely generated over~$\bbF_p$. To avoid case distinctions, define $k_0 = k = \bbC$ in the first case. Moreover, we may assume that $X$ and~$Y$ are irreducible of dimensions~$d$ and~$e$ with $d+e = g$ (if necessary, replace~$Y$ by an irreducible closed subvariety), and that we have $X_0$, $Y_0 \subset A_0$ over~$k_0$ which after extension of scalars to~$k$ give $X$, $Y \subset A$. We follow the notational convention of Remark~\ref{rem:Notation0}. Let $Z_0 \coloneqq X_0 + Y_0$ (and so $Z = X+Y$). Our goal is to show that $\sigma \colon X\times Y \to Z$ is generically finite, or equivalently, $Z=A$. We assume $Z \neq A$ and we shall derive a contradiction.

Suppose $B \subset A$ is an abelian subvariety with quotient $q \colon A \twoheadrightarrow \ol{A} = A/B$ and write $\ol{X} \coloneqq q(X)$ and $\ol{Y} \coloneqq q(Y)$. The assumption that $X$ is geometrically nondegenerate implies  $\dim(\ol{X}) + \dim(\ol{Y}) \geq \dim(\ol{A})$. By induction on~$g$, we may therefore assume  $q(Z) = \ol{X} + \ol{Y} = \ol{A}$ whenever $B \neq 0$.

For $z \in Z$, let $F_z \coloneqq \sigma^{-1}(z) \subset X\times Y$. For $n\geq 0$, let $Z_n \coloneqq \bigl\{z\in Z\bigm| \dim(F_z) \geq n\bigr\}$, which is closed in~$Z$, and let $Z_n^\circ \coloneqq Z_n\setminus Z_{n+1}$.

As in~\ref{subsec:goodopens}, we choose affine open subsets $U_X \hookrightarrow X$ and $U_Y \hookrightarrow Y$ which are smooth over~$k$. Let $\Delta_X \hookrightarrow X$ and $\Delta_Y \hookrightarrow Y$ be the closed complements, and let $\Delta = (X\times Y)\setminus (U_X\times U_Y)$. Define
\[
K \coloneqq \IC_X|_{\Delta_X}[-1]\, ,\qquad L \coloneqq \IC_Y|_{\Delta_Y}[-1]\, ,\qquad J \coloneqq \IC_{X\times Y}|_\Delta[-1]\, .
\]
which are perverse sheaves. By the same argument as in~\ref{subsec:goodopens}, $\IC_X$, $\IC_Y$, $K$ and~$L$ can be extended to constructible complexes on some model $\cA \to S$ as in~\ref{subsec:char0orp}, so that the results of Section~\ref{sec:ConvPerv} can be applied to them. We claim that there exists a dense open subset $W \subset Z$ such that the restriction of each of the objects
\begin{equation}\label{eq:fourterms}
\IC_X * \IC_Y\, ,\quad K * \IC_Y\, ,\quad \IC_X * L\, ,\quad K * L  
\end{equation}
to~$W$ is a perverse sheaf. Before we prove this, let us explain how this implies the theorem.

Suppose we have a dense open $W \subset Z$ as claimed. By assumption, $\dim(Z) = g-n = d+e-n$ for some $n\geq 1$, and then $Z_n^\circ \subset Z$ is dense open. Let $W_n^\circ \coloneqq W \cap Z_n^\circ$. We now argue as in~\ref{subsec:pfmaincalc}, to which we refer for notation. In particular, the meaning of~$F_z$, $V_z$ and~$D_z$ is the same as in that section. We find that there is a closed locus $W_n^\prime \subset W_n^\circ$ of dimension at most $d+e-2n$ such that \eqref{eq:HDvanish} holds for all $z \in W_n^\circ\setminus W_n^\prime$. With $W_n^\pprime \coloneqq \bigl\{z\in W_n^\circ \bigm| \dim(V_z) = n\bigr\}$, the same argument as in~\ref{subsec:pfmaincalc} then gives that $\dim(W_n^\pprime\setminus W_n^\prime) \leq d+e-2n$. So $W^{\circ\circ} \coloneqq W_n^\circ \setminus (W_n^\prime \cup W_n^\pprime)$ is dense open in~$Z$, and for all $z \in W^{\circ\circ}$, the fibre over~$z$ of the map $(\Delta_X \times Y) \cup (X\times \Delta_Y) \to Z$ has dimension~$n$, which is impossible for dimension reasons. This shows that we must have $\dim(Z) = g$, so $Z = A$.

It remains to be shown that there exists an open subset $W \subset Z$ as desired. We give the argument for $K * L$; exactly the same works with $K$ replaced by~$\IC_X$ or $L$ replaced by~$\IC_Y$. (For each of the terms in~\eqref{eq:fourterms}, we find an open subset~$W$ as desired, and at the end, we take the intersection of the four open subsets~$W$ that were found.) By Lemma~\ref{lem:SSPervFq}\ref{Isotyp/k1}, we may, possibly after replacing~$k_0$ by a finite extension, assume that every simple subquotient~$P_0$ of $K_0 * L_0$ has the property that~$P$ (the pullback of~$P_0$ to~$A$) is isotypical. By Propositions~\ref{prop:chiP=0} and~\ref{prop:dimStab>0}, it follows that for every simple subquotient~$P_0$ of ${}^{\mathsf{p}}H^i(K_0 * L_0)$ with $i\neq 0$, we have $\dim\bigl(\Stab(P)\bigr) > 0$, so that the identity component $B = \Stab(P)^0$ is a nonzero abelian subvariety of~$A$. If $\Sigma(P_0) \subset A_0$ is the closure of the support of~$P_0$, then $\Sigma \coloneqq \Sigma(P_0) \otimes_{k_0} k$ is stable under the action of~$B$. Writing $q \colon A \to \ol{A} = A/B$ for the quotient map and $\ol\Sigma \coloneqq q(\Sigma)$, it follows that $\Sigma = q^{-1}(\ol\Sigma)$. Moreover, $\Sigma \subset Z \subsetneq A$, whereas $q(Z) = \ol{A}$, so we cannot have $\Sigma = Z$. This shows that $P_0$ is supported on a proper closed subset of~$Z_0$. Now take $W_0 \subset Z_0$ to be the complement of the union of all $\Sigma(P_0) \subsetneq Z_0$, where $P_0$ runs through the simple subquotients of ${}^{\mathsf{p}}H^i(K_0 * L_0)$ with $i\neq 0$.
\end{proof}

\begin{remark}\label{rem:remplir}
The argument actually proves something slightly stronger: we have $X+Y = A$ whenever, with notation as above, $\dim(\ol{X}) + \dim(\ol{Y}) \geq \dim(\ol{A})$ for every quotient $q \colon A \twoheadrightarrow \ol{A} = A/B$. This condition means precisely that $(X,Y)$ `fills up~$A$' in the sense of \cite[(1.10)]{Debarre-Conn}. Thus we have the analogue of \cite[Corollaire~2.6]{Debarre-Conn} over an arbitrary base field.
\end{remark}

\begin{remark}\label{rem:ssmall}
In the situation of the theorem, when $\dim(X) + \dim(Y) = g$,   the (surjective) addition map $\sigma\colon X\times Y\to A$ is not  semismall in general (compare with Theorem~\ref{thm:semismall}; see also Remark~\ref{rem:martens} for a related result). Here is an example over an algebraically closed field of arbitrary characteristic.

Assume $A=B\times Y$, where $B$ and $Y$ are simple nonisogeneous abelian varieties, of respective dimensions $d$ and~$e$. Let $X \subset A$ be an irreducible closed $d$-dimensional subvariety. If $p\colon A\to B$ and $q\colon A\to Y$ are the two projections, $X$ is geometrically nondegenerate if and only if $p(X)=B$ and $\dim\bigl(q(X)\bigr)=\min\{d,e\}$.

Suppose $X$ is geometrically nondegenerate. The addition map~$\sigma$ is semismall if and only if, for every integer~$n\geq 0$, the codimension of $A_n \coloneqq \bigl\{a \in A \bigm| \dim(\sigma^{-1}(a)) \geq n\bigr\}$ in~$A$ is at least~$2n$. Since
\begin{equation*}
A_n= \{a\in A\mid \dim\bigl( X\cap (Y+a)\bigr)\ge n\},
\end{equation*}
this set is invariant by translation by~$Y$, and it is the inverse image under~$p$ of
\begin{equation*}
\overline A_n\coloneqq \{b\in B\mid \dim\bigl( X\cap p^{-1}(b)\bigr)\ge n\}\, .
\end{equation*}
Since $\codim(A_n,A) = \codim(\overline{A}_n,B)$, semismallness of~$\sigma$ is equivalent to semismallness of the (surjective, generically finite) projection $p\vert_X\colon X\to B$.

To construct the desired example, choose an odd integer $d\ge 3$, which we write as $d=2\ell-1$, and choose $e\ge \ell$. Let $A=B\times Y$ be as above (with $\dim(B) = d$ and $\dim(Y) = e$) and let $Z\subset \{0\}\times Y$ be a smooth subvariety of dimension~$\ell$. With $g = \dim(A) = d+e$, we have $g>2\ell$; so, by \cite[Theorem~(3.1)]{diazharb}, there exist $g-2\ell$ ample hypersurfaces in~$A$ containing~$Z$ whose intersection~$X'$ is smooth of dimension $2\ell=d+1$. Let $\cI_Z \subset \cO_{X^\prime}$ be the ideal sheaf, let $\pi \colon \tilde{X}^\prime \to X^\prime$ be the blowup of~$X^\prime$ along~$Z$, and let~$E$ be the exceptional divisor. If $D$ is a sufficiently ample divisor on~$X^\prime$ then $\pi^*(D) - E$ is very ample and $\cI_Z(D)$ is globally generated. It follows that the linear system of divisors in~$|D|$ that contain~$Z$ has base locus equal to~$Z$. Take $X = X^\prime \cap H$, for a general $H \in |D|$. By \cite[Theorem~(2.1)]{diazharb}, the singular locus of~$X$ is contained in~$Z$  and, if $X$ is singular, $\codim(X^{\sing},Z) = \codim(Z,X^\prime) = \ell$. It follows that $X^{\sing}$ is finite. 
 
Now $X$ is a connected local complete intersection of dimension $d\ge 3$ with finite singular locus, so it is irreducible. The fibre $(p\vert_X)^{-1}(0)$ contains~$Z$, hence has dimension at least $\ell >d/2$, proving that $p\vert_X$, and hence also~$\sigma$, is not semismall. Finally, $X$ is geometrically nondegenerate in~$A$ by item~(5) in Section~\ref{sec:gndeg}.
\end{remark}

As a corollary to Theorem~\ref{thm:X+Y=A}, we obtain the promised generalisation of Corollary~\ref{cor:dimZ}.

\begin{corollary}\label{cor:X1Xr}
Let $X_1,\ldots,X_r$ be closed subvarieties of an abelian variety~$A$ of dimension~$g$  over a field. Assume that at least $r-1$ of the~$X_i$ are geometrically nondegenerate. Then $\dim(X_1+\cdots+X_r) = \min\bigl\{g,\sum_{i=1}^r\, \dim(X_i) \bigr\}$.
\end{corollary}

\begin{proof} 
It suffices to prove this for $r=2$, working over an algebraically closed base field. Let $X$ and~$Y$ be closed subvarieties of~$A$ such that $X$ is geometrically nondegenerate. When $\dim(X) + \dim(Y) \ge g$, the theorem gives $X+Y=A$, whence the corollary in that case. Assume now $\dim(X) + \dim(Y) = g-n$ for some $n>0$. Let $C \subset A$ be an irreducible closed curve that generates~$A$. For every closed subvariety $V \subsetneq A$, we have $\dim(V+C) = \dim(V) + 1$. Let $Y^\prime = Y + nC$, where $nC = C + \cdots + C$ ($n$~terms). Then $\dim(X) + \dim(Y^\prime) = g$, so the theorem gives $\dim(X+Y+nC) = g$ and hence $\dim(X+Y) = g-n$.
\end{proof}

\bigskip

\noindent
Olivier Debarre, \texttt{olivier.debarre@imj-prg.fr}

\noindent
Universit\'e Paris Cit\'e and Sorbonne Universit\'e, CNRS, IMJ-PRG, F-75013 Paris, France
\medskip

\noindent
Ben Moonen, \texttt{b.moonen@science.ru.nl}

\noindent
Radboud University Nijmegen, IMAPP, Nijmegen, The Netherlands

\end{document}